\documentclass[12pt]{article}

\usepackage{amsfonts,mathrsfs}
\usepackage{amssymb,amsmath,amsbsy,amsthm}
\usepackage{dsfont}
\usepackage{blkarray}
\usepackage{tikz}
\usepackage{tkz-euclide}
\usepackage{tikz-cd}
\usetikzlibrary{patterns}
\usepackage{ae}
\usepackage{bm}
\usepackage{graphicx}
\usepackage{float}
\usepackage{cite}
 \usepackage{indentfirst}
\usepackage{lineno}
\usepackage{titlesec}
\usepackage{enumitem}
\usepackage{color}
\usepackage{aliascnt}
\usepackage{varioref}
\usepackage{hyperref}
\hypersetup{colorlinks=true,linkcolor=cyan,anchorcolor=cyan,citecolor=red,CJKbookmarks=True,
,urlcolor=cyan}
\usepackage{cleveref}

\numberwithin{equation}{section}

\def\NBC{{\rm NBC}}

\def\c{\mbox{\rm c}}

\def\And{\mbox{\rm ~and~}}

\def\I{\mbox{\rm (\hspace{0.2mm}I\hspace{0.2mm})}\,}
\def\II{\mbox{\rm (\hspace{-0.1mm}I\hspace{-0.2mm}I\hspace{-0.1mm})}\,}

\def\cal{\mathcal}
\def\udl{\underline}

\def\sst{\scriptscriptstyle}

\def\red{\color{red}}

\def\({\mbox{\rm (}}\def\){\mbox{\rm )}}

\def\b{\big}

\makeatletter

\newcommand{\Rmnum}[1]{\expandafter\@slowromancap\romannumeral #1@}
\makeatother
\newtheorem{theorem}{Theorem}[section]
\newaliascnt{lemma}{theorem}
\newtheorem{lemma}[lemma]{Lemma}
\aliascntresetthe{lemma}

\newaliascnt{proposition}{theorem}
\newtheorem{proposition}[proposition]{Proposition}
\aliascntresetthe{proposition}

\newaliascnt{fact}{theorem}

\aliascntresetthe{fact}

\newaliascnt{definition}{theorem}

\aliascntresetthe{definition}

\newaliascnt{conjecture}{theorem}

\aliascntresetthe{conjecture}

\newaliascnt{corollary}{theorem}

\aliascntresetthe{corollary}

\newaliascnt{claim}{theorem}

\aliascntresetthe{claim}

\newaliascnt{problem}{theorem}

\aliascntresetthe{problem}

\newaliascnt{remark}{theorem}
\newtheorem{remark}[remark]{Remark}
\aliascntresetthe{remark}

\newaliascnt{example}{theorem}
\newtheorem{example}[example]{Example}
\aliascntresetthe{example}

\begin{document}
\begin{center}
{\Large\bf Two Bijections on NBC Subsets}\\[7pt]
\end{center}
\vskip3mm

\begin{center}
Houshan Fu$^{1}$, Baoren Peng$^{2}$, Suijie Wang$^{3}$ and Jinxing Yang$^{4}$\\[8pt]

 $^{1}$School of Mathematics and Information Science\\
 Guangzhou University\\
Guangzhou 510006, Guangdong, P. R. China\\[12pt]

 $^{2,3,4}$School of Mathematics\\
 Hunan University\\
 Changsha 410082, Hunan, P. R. China\\[15pt]

$^{*}$Correspondence to be sent to:jxyangmath@hnu.edu.cn \\
Emails: $^{1}$fuhoushan@gzhu.edu.cn, $^{2}$pengbaoren@hnu.edu.cn, $^{3}$wangsuijie@hnu.edu.cn\\[15pt]
\end{center}

\vskip 3mm
\begin{abstract}
We establish two explicit bijections: from acyclic reorientations of an oriented matroid to no broken circuit (NBC) subsets of its underlying matroid, and from regions of a real hyperplane arrangement to its affine NBC subsets.
\vspace{1ex}\vspace{1ex}
\end{abstract}

\section{Introduction}\label{sec1}
In order to solve the long-standing four colour problem, Birkhoff  \cite{Birkhoff1912} first introduced the chromatic polynomial of planar graphs  in 1912. Twenty years later, Whitney \cite{Whitney1932} defined the chromatic polynomial for finite graphs and gave a combinatorial interpretation for the coefficients of the chromatic polynomial. This interpretation is called  Whitney's broken circuit theorem. In 1973, applying Whitney's broken circuit theorem,  Stanley \cite[Proposition 1.1]{Stanley1973} obtained that the number of acyclic orientations of a finite graph $G$ equals the number of no broken circuits of the edge set $E(G)$, called Stanley's theorem.

As an extension of oriented graphs,  in 1978 Bland and Las Vergnas \cite{Bland-Lasvergnas1978} introduced the concept of oriented matroids.
Let $\mathcal{M}=\b(E,\mathcal{C}\b)$ be an oriented matroid on the ground set $E$.  For $A\subseteq E$, denote by ${{}_{-A}}\mathcal{M}=\b(E,{{}_{-A}}\mathcal{C}\b)$ the {\em reorientation} of  $\mathcal{M}$ by reversing signs on $A$.  By a slight abuse of language, we will usually say that $A\subseteq E$ is an {\em acyclic reorientation} of $\cal{M}$ if  ${}_{-A}\mathcal{M}$ is acyclic (i.e., containing no positive circuits).  Denote by $A(\mathcal{M})$ the collection of subsets $A$ of $E$ such that $A$ is an acyclic reorientation of $\cal{M}$. Suppose the underlying matroid of $\cal{M}$ is $\udl{\cal{M}}=(E,\udl{\cal{C}})$ and the ground set $E$ is totally ordered by $\prec$. A {\em broken circuit} of $\udl{\cal{M}}$ is a subset of $E$ obtained from a member of $\udl{\cal{C}}$ by deleting its maximal element with respect to $\prec$.
A subset $A\subseteq E$ is an {\em NBC subset} of $\udl{\cal{M}}$ if $A$ contains no broken circuits of $\udl{\cal{M}}$. Denote by ${\rm NBC}(\udl{\cal{M}})$ the collection of NBC subsets of $\udl{\cal{M}}$. As a generalization of the chromatic polynomial, in 1964 Rota \cite{Rota1964} introduced the characteristic polynomial for matroids (or combinatorial geometries) and obtained a matroid version of Whitney's broken circuit theorem, see \cite{Bjorner1991,Brylawski-Oxley1981,Jambu-Terao1986,Sagan1995,Stanley2007} for details. In 1980, Stanley's theorem  was extended to oriented matroids by Las Vergnas \cite{LasVergnas1980} .

\begin{theorem}[\cite{LasVergnas1980}, Theorem 3.1]\label{Acyclic-NBC}
If $\mathcal{M}$ is an oriented matroid and $\underline{\mathcal{M}}$ its underlying matroid, then \[
|A(\mathcal{M})|=|\NBC(\underline{\mathcal{M}})|.
\]
\end{theorem}

A {\em hyperplane arrangement} $\mathcal{A}$  is a finite collection of (affine) hyperplanes in an $n$-dimensional vector space $V$.
In particular, $\mathcal{A}$ is called a  {\em linear arrangement} if all hyperplanes in $\mathcal{A}$ pass through the origin. Given a total ordering $\prec$ on $\mathcal{A}$, an {\em affine circuit} is a smallest subset $\mathcal{B}$ of $\mathcal{A}$ such that $\cap\mathcal{B}:=\bigcap_{H\in \mathcal{B}}H\ne \emptyset$ and $\dim(\cap\mathcal{B})+|\mathcal{B}|=n+1$. An {\em affine broken circuit} is a subset of $\mathcal{A}$ obtained from an affine circuit by removing its maximal element with respect to  $\prec$. A subset $\mathcal{B}$ of $\mathcal{A}$ is called an {\em affine NBC subset} if $\mathcal{B}$ contains no affine broken circuits and $\cap\mathcal{B}\neq \emptyset$. Let $\NBC(\mathcal{A})$ be the collection of all affine NBC subsets of $\mathcal{A}$. When $\mathcal{A}$ is a linear arrangement, `affine circuit',  `affine broken circuit' and `affine NBC subset' are referred to as `circuit', `broken circuit' and `NBC subset' respectively. In 1992, Orlik-Terao \cite{Orlik1992} extended Whitney's broken circuit theorem to  hyperplane arrangements, that is, for $k=0,1,\ldots,n$, the absolute value of the coefficient of $t^{n-k}$ in the characteristic polynomial $\chi(\mathcal{A},t)$ of $\mathcal{A}$ equals the number of all affine NBC $k$-subsets of $\mathcal{A}$. When $V$ is a real vector space, the complement $V-\bigcup_{H\in \mathcal{A}}H$ consists of finitely many connected components, called the {\em regions} of $\mathcal{A}$. Let $\mathcal{R}(\mathcal{A})$ be the collection of all regions of $\mathcal{A}$. In 1975, Zaslavsky \cite{Zaslavsky1975} showed that the number of all regions of $\mathcal{A}$ is exactly equal to $|\chi(\mathcal{A},-1)|$. Below is an easy consequence of Zaslavsky's formula and Whitney's broken circuit theorem for  hyperplane arrangements, which is a version of Stanley's theorem for real hyperplane arrangements.
\begin{theorem}[ \cite{Zaslavsky1975,Orlik1992}]\label{Region-NBC}If $\mathcal{A}$ is a hyperplane arrangement in a real vector space, then
\begin{equation*}
|\mathcal{R}(\mathcal{A})|=|\NBC(\mathcal{A})|.
\end{equation*}
\end{theorem}
Bijective proofs for \autoref{Acyclic-NBC} and \autoref{Region-NBC} have been well studied over the past 25 years. By this time there were many important bijections constructed inductively by the deletion-contraction operation of hyperplane arrangements or oriented matroids as well. For example, Jewell-Orlik \cite{Jewell-Orlik2002} established the bijection from $\mathcal{R}(\mathcal{A})$ to $\NBC(\mathcal{A})$, which was reworked by Delucchi  \cite{Delucchi2008} in classical settings of hyperplane arrangement complexes. In a series of papers \cite{Gioan2002,Gioan-Vergnas2006,Gioan-Vergnas2007,Gioan-Vergnas2009, Gioan-Vergnas2018,Gioan-Vergnas2019}, Gioan and Las Vergnas made a systematic study on a framework to provide bijections for acyclic orientations of graphs and matroids, regions of hyperplane arrangements, no broken circuits, etc.. The first chapter of Gioan's thesis \cite{Gioan2002} provides a deletion-contraction framework for oriented matroids whose specification are available for hyperplane arrangements. This framework is detailed precisely in \cite{Gioan-Vergnas2019} for graphs. All these are developed in the celebrated ``active bijection" theory, initially defined in \cite{Gioan2002} and detailed in \cite{Gioan-Vergnas2018}. Indeed, the refined active bijection produces bijections from acyclic reorientations or regions to no broken circuits.  Furthermore,  the canonical active bijection available for oriented matroids has the advantage of preserving Tutte polynomial activities and being substantially independent of the initial orientation.

Note that the bijection in \cite{Jewell-Orlik2002, Delucchi2008} are defined by the deletion-contraction operation, which means the correspondence are not characterized specifically, but built step by step in a recursive way. A bijection for graphs by Blass and Sagan \cite{Blass-Sagan1986} in 1986 is particularly noteworthy. In the Blass-Sagan's algorithm, the correspondence of each step is stated in a closed form and its domain (or its range as well) is characterized explicitly. In this paper, the Blass-Sagan's bijection available for graphs will be extended to oriented matroids and real hyperplane arrangements, which gives bijective proofs for \autoref{Acyclic-NBC} and \autoref{Region-NBC} respectively. As a straightforward generalization, our bijection will be established more explicitly by introducing the pairings of NBC subsets and acyclic reorientation sets of oriented matroids (or regions of real hyperplane arrangements), see \autoref{Sec2.2} and \autoref{sec3} for details. 

The remainder of this paper is organized as follows. \autoref{sec4} is devoted to establishing a bijection between all acyclic reorientations of an oriented matroid and its NBC subsets. Applying this bijection to real hyperplane arrangements,  \autoref{sec3} gives a bijection from all regions of a real hyperplane arrangement to its affine NBC subsets.
\section{Bijection from acyclic reorientations to NBC subsets}\label{sec4}
This section aims to establish an explicit bijection from acyclic reorientations of an oriented matroid to no broken circuit (NBC) subsets of its underlying matroid.
\subsection{Basics on oriented matroids}
We start with some basics on matroids and oriented matroids following from \cite{Oxley2011,Bjorner1999}. In 1935, Whitney \cite{Whitney1935} introduced matroids to capture abstract properties of vector dependences and  graph cycles. A {\em matroid} $M=(E, C)$ is an ordered pair of a finite set $E$ and  a collection $C$ of subsets of
$E$ satisfying the following three properties:
\begin{itemize}
  \item [{(C1)}] $\emptyset\notin C$;
  \item [{(C2)}] if $X,Y\in C$ and $X\subseteq Y$, then $X=Y$;
  \item [{(C3)}] if $X,Y$ are distinct members of $C$ and $e\in X\cap Y$, then there is a member $Z\in C$ such that $Z\subseteq X\cup Y-\{e\}$.
\end{itemize}
Members of $ C$ are called circuits of $ M$ and $E$ is called the ground set of $ M$. In particular, a single element of $E$ is a loop if itself is a circuit. A matroid is said to be loopless if it contains no loops. A subset of $E$ is independent if it contains no circuits of $ M$ and dependent otherwise. Additionally, a subset $X$ of $E$ is called a flat of $ M$ if for any element $e\in E-X$, there is no circuit $Y$ such that $e\in Y\subseteq X\cup\{e\}$.  Let $ M=(E, C)$ be a matroid on the ground set $E$ and $X$ a subset of $E$. The deletion $ M\backslash X=(E-X, C\backslash X)$ is a matroid on the ground set $E-X$ with the circuit set
\[
 C\backslash X:=\{Y\mid Y\in C\And Y\subseteq E-X\}.
\]
The contraction $ M/X=(E-X, C/X)$ is a matroid on the ground set $E-X$ with the circuit set
\[
 C/X:={\rm min}\{Y-X\mid Y\in C\And Y-X\ne\emptyset\},
\]
where `${\rm min}$' means the inclusion-minimal non-empty sets. In addition, a minor of $ M$ is a matroid obtained from $ M$ by a sequence of deletions and contractions.

To equip matroids with orientations, in 1978 Bland and Las Vergnas \cite{Bland-Lasvergnas1978} introduced the concept of oriented matroids.  For any two disjoint subsets $X^+$  and $X^-$ of a finite set $E$, the paring $X=(X^+,X^-)$ is said to be a signed subset of $E$ and $X^+$ (resp. $X^-$) is the set  of positive (resp. negative) elements of $X$. The support of $X$ is $\underline{X}=X^+\cup X^-$, and the opposite of $X$ is $-X=(X^-,X^+)$.  An {\em oriented matroid} $\mathcal{M}=(E,\mathcal{C})$ is an ordered pair of a finite set $E$ and  a collection  $\mathcal{C}$ of signed subsets of $E$ satisfying:
\begin{itemize}
  \item [{($\cal{C}1$)}] $\emptyset\notin\mathcal{C}$;
  \item [{($\cal{C}2$)}] $\mathcal{C}=-\mathcal{C}$;
  \item [{($\cal{C}3$)}] if $X,Y\in\mathcal{C}$ and $\underline{X}\subseteq \underline{Y}$, then $X=Y$ or $X=-Y$;
  \item [{($\cal{C}4$)}]if two distinct members $X,Y\in\mathcal{C}$, $X\ne-Y$ and $e\in X^+\cap Y^-$, then there is $Z\in\mathcal{C}$ such that $Z^+\subseteq X^+\cup Y^+-\{e\}$ and $Z^-\subseteq X^-\cup Y^--\{e\}$.
\end{itemize}
Members of $\mathcal{C}$ are called {signed circuits} of $\mathcal{M}$ and $E$ is called the {ground set} of $\mathcal{M}$. Let $\underline{\mathcal{C}}=\{\underline{X}\mid X\in\mathcal{C}\}$. Then the ordinary matroid $\underline{\mathcal{M}}=(E,\underline{\mathcal{C}})$ is said to be the underlying matroid of $\mathcal{M}$. Let $\mathcal{M}=\b(E,\mathcal{C}\b)$ be an oriented matroid on the ground set $E$. A signed circuit $X=(X^+,X^-)\in\mathcal{C}$ is called positive if $X^-=\emptyset$. An oriented matroid is said to be acyclic if it contains no positive circuits. For $A\subseteq E$, ${{}_{-A}}\mathcal{M}=\b(E,{{}_{-A}}\mathcal{C}\b)$ is a reorientation of  $\mathcal{M}$, where ${{}_{-A}}\mathcal{C}$ consists of signed circuits ${}_{-A}X=({}_{-A}X^+,{}_{-A}X^-)$ with ${}_{-A}X^+=(X^+-A)\cup(X^-\cap A)$ and ${}_{-A}X^-=(X^--A)\cup(X^+\cap A)$. Our first result is to establish a bijection between acyclic reorientations of $\cal{M}$ and NBC subsets of $\udl{\cal{M}}$. The deletion and contraction operations of ordinary matroids can be naturally extended to oriented matroids. More precisely, for $X\subseteq E$, the deletion $\mathcal{M}\backslash X=(E-X,\mathcal{C}\backslash X)$ has the signed circuit set
\[
\mathcal{C}\backslash X:=\{Y\mid Y\in\mathcal{C}\And\underline{Y}\subseteq E-X\},
\]
and the contraction $\mathcal{M}/X=(E-X,\mathcal{C}/X)$ has
\[
\mathcal{C}/X:={\rm min}\{Y\backslash X\mid Y\in\mathcal{C}\And \underline{Y}-X\ne\emptyset\},
\]
where $Y\backslash X=(Y^+-X,Y^--X)$. Likewise, a minor of $\mathcal{M}$ is obtained from $\mathcal{M}$ by a sequence of deletions and contractions. Note from \cite[3.3.3 Proposition]{Bjorner1999} that for disjoint subsets $X,Y$ of $E$,
\[
\mathcal{M}\backslash X/Y=\mathcal{M}/Y\backslash X,
\]
i.e., the minor of $\mathcal{M}$ does not depend on the ordering of deletion and contraction operations. Clearly, $\underline{\mathcal{M}\backslash X/Y}=\underline{\mathcal{M}}\backslash X/Y$.

\subsection{Bijection for oriented matroids}\label{Sec2.2}
Given an oriented matroid $\cal{M}=(E,\cal{C})$ with $E=\{e_1,e_2,\ldots,e_m\}$. Recall
\[
A(\mathcal{M})=\{A\subseteq E\mid {{}_{-A}}\mathcal{M} {\rm ~is~acyclic}\}\; \And\; \NBC(\udl{\cal{M}})=\{N\subseteq E\mid N {\rm ~is~NBC~subset~of~}\udl{\cal{M}}\}.
\]
In this section, the Blass-Sagan's algorithm will be extended to a bijection from $A(\mathcal{M})$ to $\NBC(\udl{\cal{M}})$. In the Blass-Sagan's algorithm, the key ingredient is the set $\mathscr{D}_k$ of all possible results of its first $k$ steps for $k=1,\ldots,m$. Likewise, below we will introduce the key set $\mathscr{N}_k$ and the bijection $\psi_k:\mathscr{N}_{k-1}\to \mathscr{N}_{k}$ whose composition will automatically lead to a bijection from $A(\mathcal{M})$ to $\NBC(\udl{\cal{M}})$. Note that if $\cal{M}$ has loops, then $\cal{M}$ has no acyclic reorientations and $\NBC(\udl{\cal{M}})=\emptyset$. We may assume that $\cal{M}$ is loopless and the ground set $E$ is totally ordered by $e_1\prec e_2\prec\cdots \prec e_m$. For $k=0,1,\ldots,m$, let $E_k=\{e_1,e_2,\ldots,e_k\}$. Note that if $N_k\subseteq E_k$ and $N_k^c:=E_k-N_k$, then the minor $\mathcal{M}\backslash N_k^c/N_k$ has the ground set $E-E_k$. Define $\mathscr{N}_k$ to be the set of pairings $(N_k,A_k)$ with $N_k\subseteq E_k$, $A_k\subseteq E-E_k$ satisfying
\begin{itemize}\label{bijec-1}
  \item [\I]  $N_k\in\NBC(\udl{\cal{M}})$,
  \item [\II] $\mathcal{M}_{k}:={}_{-A_k}\b(\mathcal{M}\backslash N_k^c/N_k\b)$ is acyclic.
\end{itemize}
The map $\psi_k:\mathscr{N}_{k-1}\to\mathscr{N}_{k}$ is defined as follows,
\begin{eqnarray}\label{psi_k}
\psi_k\b(N_{k-1},A_{k-1}\b)=
\begin{cases}
\b(N_{k-1}\cup\{e_{k}\},A_{k-1}\b) & \mbox{if\;} e_k\notin A_{k-1}\mbox{\;and\;}{}_{-e_k}\mathcal{M}_{k-1}\mbox{\;is acyclic}; \\
\b(N_{k-1},A_{k-1}\b) & \mbox{if\;} e_k\notin A_{k-1}\mbox{\;and\;}{}_{-e_k}\mathcal{M}_{k-1}\mbox{\;is not acyclic}; \\
\b(N_{k-1},A_{k-1}-\{e_k\}\b) & \mbox{if\;} e_k\in A_{k-1}.
\end{cases}
\end{eqnarray}
From \autoref{lemma8}, \autoref{lemma9} and \autoref{lemma10}, we will see that $\psi_k$ is well-defined, injective and surjective respectively, i.e., each $\psi_k$ is a bijection for $k=1,2,\ldots,m$. It follows that the composition $\psi:=\psi_m\circ\psi_{m-1}\circ\cdots\circ\psi_1$ is a bijection from $\mathscr{N}_0$ to $\mathscr{N}_m$. Note
\[
\mathscr{N}_0=\b\{(\emptyset,A_0)\mid A_0\in A(\mathcal{M})\b\}\quad\And\quad \mathscr{N}_m=\b\{(N_m,\emptyset)\mid N_m\in\NBC(\underline{\mathcal{M}})\b\}.
\] Immediately,  the map $\psi$ induces a bijection from $A(\mathcal{M})$ to $\NBC(\underline{\mathcal{M}})$. The example below is a glimpse of this bijection.
\begin{example}
{\rm  Suppose $\cal{M}=(E,\cal{C})$ is an oriented matroid with the ground set
 $E=\{e_1,e_2,e_3,e_4\}$  and $\cal{C}=\{\pm X_i\mid i=1,2,3,4\}$, where
\[
X_1=(e_1,e_2e_3),\quad X_2=(e_1e_2,e_4), \quad X_3=(e_1,e_3e_4),\quad X_4=(e_2e_3,e_4).
\]
Without causing confusion, we use the sequence $e_{i}e_j\cdots$ to denote the set $\{e_i,e_j,\ldots\}$. By routine analyses, the set of all acyclic reorientations  of $\cal{M}$ is
\[
A(\mathcal{M})=\{\emptyset,\,e_2,\,e_3,\,e_1e_3,\,e_2e_4,\,e_1e_2e_4,\,e_1e_3e_4,\,e_1e_2e_3e_4\}.
\]
It is clear that the underlying matroid $\udl{\cal{M}}$ is the uniform matroid of rank 2 and size 4. Given the total order $e_1\prec e_2\prec e_3\prec e_4$. The collection of all subsets of $E$ containing no broken circuits of $\udl{\cal{M}}$ is
\[
\NBC(\underline{\mathcal{M}})=\{\emptyset,\,e_1,\,e_2,\,e_3,\,e_4,\,e_1e_4,\,e_2e_4,\,e_3e_4\}.
\]
The specific bijection from $A(\mathcal{M})$ to $\NBC(\underline{\mathcal{M}})$ is presented below.
\begin{eqnarray*}
\begin{array}{cccccccccccccc}
A(\mathcal{M})     &\hspace{-2mm}\xrightarrow{}     &\hspace{-2mm}\mathscr{N}_0     &\hspace{-2mm}\xrightarrow{{\red\psi_1}}     &\hspace{-2mm}\mathscr{N}_1     &\hspace{-2mm}\xrightarrow{{\red\psi_2}}     &\hspace{-2mm}\mathscr{N}_2     &\hspace{-2mm}
\xrightarrow{{\red\psi_3}}     &\hspace{-2mm}\mathscr{N}_3     &\hspace{-2mm}\xrightarrow{{\red\psi_4}}     &\hspace{-2mm}\mathscr{N}_4     &\hspace{-2mm}\xrightarrow{}     &\hspace{-2mm}\NBC(\underline{\mathcal{M}})\\
\emptyset     &\hspace{-2mm}\mapsto     &\hspace{-2mm}(\emptyset,\emptyset)     &\hspace{-2mm}\mapsto     &\hspace{-2mm}(\emptyset,\emptyset)     &\hspace{-2mm}\mapsto     &\hspace{-2mm}(e_2,\emptyset)     &\hspace{-2mm}
\mapsto     &\hspace{-2mm}(e_2,\emptyset)     &\hspace{-2mm}\mapsto     &\hspace{-2mm}(e_2e_4,\emptyset)     &\hspace{-2mm}\mapsto     &\hspace{-2mm}e_2e_4\\
e_2     &\hspace{-2mm}\mapsto     &\hspace{-2mm}(\emptyset,e_2)     &\hspace{-2mm}\mapsto     &\hspace{-2mm}(\emptyset,e_2)     &\hspace{-2mm}\mapsto     &\hspace{-2mm}(\emptyset,\emptyset)     &\hspace{-2mm}
\mapsto     &\hspace{-2mm}(e_3,\emptyset)     &\hspace{-2mm}\mapsto     &\hspace{-2mm}(e_3e_4,\emptyset)     &\hspace{-2mm}\mapsto     &\hspace{-2mm}e_3e_4\\
e_3     &\hspace{-2mm}\mapsto     &\hspace{-2mm}(\emptyset,e_3)     &\hspace{-2mm}\mapsto     &\hspace{-2mm}(e_1,e_3)     &\hspace{-2mm}\mapsto     &\hspace{-2mm}(e_1,e_3)     &\hspace{-2mm}
\mapsto     &\hspace{-2mm}(e_1,\emptyset)     &\hspace{-2mm}\mapsto     &\hspace{-2mm}(e_1e_4,\emptyset)     &\hspace{-2mm}\mapsto     &\hspace{-2mm}e_1e_4\\
e_1e_3     &\hspace{-2mm}\mapsto     &\hspace{-2mm}(\emptyset,e_1e_3)     &\hspace{-2mm}\mapsto     &\hspace{-2mm}(\emptyset,e_3)     &\hspace{-2mm}\mapsto     &\hspace{-2mm}(\emptyset,e_3)     &\hspace{-2mm}
\mapsto     &\hspace{-2mm}(\emptyset,\emptyset)     &\hspace{-2mm}\mapsto     &\hspace{-2mm}(e_4,\emptyset)     &\hspace{-2mm}\mapsto     &\hspace{-2mm}e_4\\
e_2e_4     &\hspace{-2mm}\mapsto     &\hspace{-2mm}(\emptyset,e_2e_4)     &\hspace{-2mm}\mapsto     &\hspace{-2mm}(e_1,e_2e_4)     &\hspace{-2mm}\mapsto     &\hspace{-2mm}(e_1,e_4)     &\hspace{-2mm}
\mapsto     &\hspace{-2mm}(e_1,e_4)     &\hspace{-2mm}\mapsto     &\hspace{-2mm}(e_1,\emptyset)     &\hspace{-2mm}\mapsto     &\hspace{-2mm}e_1\\
e_1e_2e_4     &\hspace{-2mm}\mapsto     &\hspace{-2mm}(\emptyset,e_1e_2e_4)     &\hspace{-2mm}\mapsto     &\hspace{-2mm}(\emptyset,e_2e_4)     &\hspace{-2mm}\mapsto     &\hspace{-2mm}(\emptyset,e_4)     &\hspace{-2mm}
\mapsto     &\hspace{-2mm}(e_3,e_4)     &\hspace{-2mm}\mapsto     &\hspace{-2mm}(e_3,\emptyset)     &\hspace{-2mm}\mapsto     &\hspace{-2mm}e_3\\
e_1e_3e_4     &\hspace{-2mm}\mapsto     &\hspace{-2mm}(\emptyset,e_1e_3e_4)     &\hspace{-2mm}\mapsto     &\hspace{-2mm}(\emptyset,e_3e_4)     &\hspace{-2mm}\mapsto     &\hspace{-2mm}(e_2,e_3e_4)     &\hspace{-2mm}
\mapsto     &\hspace{-2mm}(e_2,e_4)     &\hspace{-2mm}\mapsto     &\hspace{-2mm}(e_2,\emptyset)     &\hspace{-2mm}\mapsto     &\hspace{-2mm}e_2\\
e_1e_2e_3e_4     &\hspace{-2mm}\mapsto     &\hspace{-2mm}\hspace{-1mm}(\emptyset,e_1e_2e_3e_4)     &\hspace{-2mm}\mapsto     &\hspace{-2mm}(\emptyset,e_2e_3e_4)     &\hspace{-2mm}\mapsto     &\hspace{-2mm}(\emptyset,e_3e_4)     &\hspace{-2mm}
\mapsto     &\hspace{-2mm}(\emptyset,e_4)     &\hspace{-2mm}\mapsto     &\hspace{-2mm}(\emptyset,\emptyset)     &\hspace{-2mm}\mapsto     &\hspace{-2mm}\emptyset
\end{array}
\end{eqnarray*}
}
\end{example}
To prove that $\psi_k$ is a bijection, we need the following proposition
\begin{proposition}\label{acyclic-contraction}
Let $\mathcal{M}=(E,\mathcal{C})$ be an oriented matroid, $e\in E$, and $(e,\emptyset)\notin\mathcal{C}$. If $\mathcal{M}/\{e\}$ is acyclic, then $\mathcal{M}$ is acyclic.
\end{proposition}
\begin{proof}
For convenience, we denote by $\mathcal{M}'=\mathcal{M}/\{e\}$. It follows that the set of signed circuits of $\mathcal{M}'$ is
\begin{equation*}\label{signcircuit1}
\cal{C}'={\rm min}\b\{X\backslash\{e\}\mid  X \in \mathcal{C}\And\underline{X}-\{e\}\ne\emptyset\b\}.
\end{equation*}
From the axiom $(\cal{C}3)$ on signed circuits, we have an easy fact that for any $X \in \mathcal{C}$ with $\{e\}\subsetneqq\underline{X}$, we have $X\backslash\{e\}\in \cal{C}'$.
Suppose $\mathcal{M}$ is not acyclic, i.e.,  $\mathcal{M}$ has a positive circuit $Y_0=(\underline{Y_0},\emptyset)$. The above fact implies that $e\notin\underline{Y_0}$. Otherwise, $Y_0\backslash\{e\}$ is a positive circuit of $\mathcal{M}'$, a contradiction with $\mathcal{M}'$ being acyclic.  Let
\[
\mathcal{C}_{e,\sst Y_0}:=\b\{Y\in\mathcal{C}\mid e\in \underline{Y}\And \underline{Y}-\{e\}\subseteq \underline{Y_0}\b\}.
\]

Firstly, we claim $\mathcal{C}_{e,\sst Y_0}\ne \emptyset$. Since $Y_0\in \cal{C}$  and $e\notin\underline{Y_0}$, the definition of $\cal{C}'$ implies that there is a signed circuit $Y_1\in \cal{C}$ such that $Y_1\backslash\{e\}\in \cal{C}'$ and $\udl{Y_1}- \{e\}\subseteq \udl{Y_0}$. It remains to prove $e\in \udl{Y_1}$. If not,  we have $Y_1\in \cal{C}$ and $\udl{Y_1}\subseteq \udl{Y_0}$. The axiom $(\cal{C}3)$ on signed circuits of $\cal{M}$ implies $Y_1=Y_0$ or $-Y_0$, i.e., $Y_0\in \cal{C}'$, a contradiction with $\mathcal{M}'$ being acyclic.

Secondly, we claim that $\mathcal{C}_{e,\sst Y_0}$ contains a signed circuit with only one negative element $e$. Note that $Y\in \mathcal{C}_{e,\sst Y_0}\Leftrightarrow -Y\in \mathcal{C}_{e,\sst Y_0}$.    We may assume $Y_1\in \mathcal{C}_{e,\sst Y_0}$ and  $e\in Y_1^-$. Clearly, the claim is true if $Y_1^-=\{e\}$ . If $Y_1^-$ contains a distinct element $f\ne e$, obviously we have $f\in Y_0^+$. The axiom $(\cal{C}4)$ implies that there exists a signed circuit $Y_2$ of $\cal{M}$ satisfying
\[
Y_2^+\subseteq Y_0^+\cup Y_1^+-\{f\}=\underline{Y_0}-\{f\}\quad\And\quad Y_2^-\subseteq Y_0^-\cup Y_1^--\{f\} =Y_1^--\{f\},
\]
which means $\underline{Y_2}-\{e\}\subsetneqq\underline{Y_0}$. Applying the axiom $(\cal{C}3)$ again, we have $e\in\underline{Y_2}$. Consequently, we obtain $Y_2\in \mathcal{C}_{e,\sst Y_0}$ with less negative elements than $Y_1$. Continuing this process, after finitely many steps it will produce a signed circuit in $\mathcal{C}_{e,\sst Y_0}$ with only one negative element $e$. So the claim is verified.

Now we may assume $Y=(Y^+,e)\in \mathcal{C}_{e,\sst Y_0}$. The assumption $(e,\emptyset)\notin\mathcal{C}$ implies $Y^+\ne \emptyset$. From the fact stated at the very beginning, we have $Y\backslash\{e\}=(Y^+,\emptyset)\in \cal{C}'$, a contradiction with $\cal{M}'$ being acyclic. This completes the proof.
\end{proof}

\subsection{Proofs on the bijection $\psi_k$}
\begin{lemma}\label{lemma8}
$\psi_k$ is well-defined.
\end{lemma}
\begin{proof}
Given a paring $(N_{k-1},A_{k-1})\in\mathscr{N}_{k-1}$ and let $(N_k,A_k)=\psi_k(N_{k-1},A_{k-1})$. From \eqref{psi_k}, we have $N_k\subseteq E_k$ and $A_k\subseteq E-E_k$ clearly. To prove $(N_k,A_k)\in \mathscr{N}_{k}$, it is enough to show that $N_k\in\NBC(\udl{\cal{M}})$ and $\mathcal{M}_{k}:={}_{-A_k}\b(\mathcal{M}\backslash N_k^c/N_k\b)$ is acyclic. We proceed the proof in the following three cases.

\vspace{3mm}{\rm Case 1.} If $e_k\notin A_{k-1}$ and ${}_{-e_k}\mathcal{M}_{k-1}$ is acyclic, then
\[
N_k=N_{k-1}\cup\{e_k\} \quad\And\quad A_k=A_{k-1}.
\]
Since $(N_{k-1},A_{k-1})\in\mathscr{N}_{k-1}$, we have $N_{k-1}\in\NBC(\underline{\mathcal{M}})$ which implies $N_k$ contains no circuits of $\udl{\cal{M}}$. Suppose $N_k$ contains broken circuits of $\udl{\cal{M}}$, i.e., $N_k\notin \NBC(\underline{\mathcal{M}})$. There is a circuit $X\in \udl{\cal{C}}$ with the maximal element $e_l$ such that $X\subseteq N_k\cup\{e_l\}$. Moreover, we have $l>k$ and $e_k\in X$. Otherwise if $l\le k$ (or $e_k\notin X$), then $N_{k-1}$ contains the broken circuit $X-{e_l}$ which contradicts $N_{k-1}\in\NBC(\udl{\cal{M}})$. Recall the definitions of deletion and contraction operations that the set of circuits of $\underline{\mathcal{M}_{k-1}}= \underline{\mathcal{M}}\backslash N_{k-1}^c/N_{k-1}$ is
\[
{\rm min}\b\{Y-N_{k-1}\mid Y\in\underline{\mathcal{C}},\, Y\subseteq E-N_{k-1}^c\And Y-N_{k-1}\ne\emptyset\b\}.
\]
Since $\mathcal{M}_{k-1}$ is acyclic, neither $e_k$ nor $e_l$ is a loop of  $\underline{\mathcal{M}_{k-1}}$. It follows that $X-N_{k-1}=\{e_k,e_l\}$ is a circuit of $\underline{\mathcal{M}_{k-1}}$. Then either  $(e_le_k,\emptyset)$ or $(e_l,e_k)$ is a signed circuit of $\mathcal{M}_{k-1}$,  which is impossible since both $\mathcal{M}_{k-1}$ and ${}_{-e_k}\mathcal{M}_{k-1}$ are acyclic. So  we have $N_k\in\NBC(\underline{\mathcal{M}})$. It remains to show that $\mathcal{M}_k$ is acyclic. Since $\mathcal{M}_k={}_{-A_k}\b(\mathcal{M}\backslash N_k^c/N_k\b)$, $N_k=N_{k-1}\cup\{e_k\}$, and $A_k=A_{k-1}$, we have
$\mathcal{M}_k=\mathcal{M}_{k-1}/\{e_k\}$, which means that the set of signed circuits of $\mathcal{M}_k$ is
\begin{equation*}\label{circuit1}
{\rm min}\b\{Y\backslash\{e_k\}\mid Y{\rm~is~a~signed~circuit~of~}\mathcal{M}_{k-1} \And\underline{Y}-\{e_k\}\ne\emptyset\b\}.
\end{equation*}
Suppose $\mathcal{M}_k$ is not acyclic, i.e., $\mathcal{M}_k$ has a positive circuit, say $Z=(\underline{Z},\emptyset)$. Then $Z=Y\backslash\{e_k\}$ for some signed circuit $Y$ of $\mathcal{M}_{k-1}$. Note that $\mathcal{M}_{k-1}$ is acyclic, which implies $e_k\in \udl{Y}$. So we have $Y=(\underline{Z},e_k)$ clearly, i.e., the oriented matroid ${}_{-e_k}\mathcal{M}_{k-1}$ has a positive circuit ${}_{-e_k}Y=\b(\underline{Z}\cup\{e_k\},\emptyset\b)$, contradicting that  ${}_{-e_k}\mathcal{M}_{k-1}$ is acyclic. So we have $(N_k,A_k)\in\mathscr{N}_k$ in this case.

\vspace{3mm}{\rm Case 2.}  If $e_k\notin A_{k-1}$ and ${}_{-e_k}\mathcal{M}_{k-1}$ is not acyclic, then
\[
N_k=N_{k-1}\quad \And \quad A_k=A_{k-1}.
\]
It follows that $\mathcal{M}_k={}_{-A_k}\b(\mathcal{M}\backslash N_k^c/N_k\b)=\mathcal{M}_{k-1}\backslash\{e_k\}$ and each positive circuit of $\mathcal{M}_k$ is also a positive circuit of $\mathcal{M}_{k-1}$. Note that $\mathcal{M}_{k-1}$  contains no positive circuits, so does $\mathcal{M}_k$. Namely, $\mathcal{M}_k$ is acyclic. Clearly, we have $N_k=N_{k-1}\in\NBC(\underline{\mathcal{M}})$. So $(N_k,A_k)\in\mathscr{N}_k$.

\vspace{3mm}{\rm Case 3.}   If $e_k\in A_{k-1}$, then
\[
N_k=N_{k-1}\quad \And \quad A_k=A_{k-1}-\{e_k\}.
\]
Note that the element $e_k$ is not in the ground set of $\mathcal{M}_k$, which implies $\mathcal{M}_k=\mathcal{M}_{k-1}\backslash\{e_k\}$. By similar arguments as Case 2, we have $(N_k,A_k)\in\mathscr{N}_k$.
\end{proof}

\begin{lemma}\label{lemma9}
$\psi_k$ is injective.
\end{lemma}
\begin{proof}If $\psi_k(N_{k-1},A_{k-1})=(N_k,A_k)$, the definition \eqref{psi_k} implies $N_{k-1}=N_k-\{e_k\}$ and $A_k=A_{k-1}-\{e_k\}$.
Suppose there is another element $(N_{k-1}',A_{k-1}')$ satisfying $\psi_k(N_{k-1}',A_{k-1}')=(N_k,A_k)$. Obviously, $N_{k-1}'=N_{k-1}=N_k-\{e_k\}$ and
\[
A_k=A_{k-1}-\{e_k\}=A_{k-1}'-\{e_k\}.
\]
Suppose $A_{k-1}\neq A_{k-1}'$ and assume  $e_k\in A_{k-1}$, $e_k\notin A_{k-1}'$, i.e., $A_{k-1}=A_k\sqcup\{e_k\}$ and $A_{k-1}'=A_k$. Notice $e_k\in A_{k-1}$ corresponds to the 3rd case in  \eqref{psi_k}, which implies $N_{k-1}=N_k$ and $(N_k,A_k)=(N_{k-1}',A_{k-1}')$.  From the 2nd case of \eqref{psi_k}, the oriented matroid
\[
{}_{-A_{k-1}'\sqcup\{e_k\}}\b(\mathcal{M}\backslash(N_{k-1}')^c/N_{k-1}'\b) ={}_{-A_{k-1}}\b(\mathcal{M}\backslash N_{k-1}^c/N_{k-1}\b)
\]
is not acyclic, a contradiction with $\mathcal{M}_{k-1}$ being acyclic.
\end{proof}

\begin{lemma}\label{lemma10}
$\psi_k$ is surjective.
\end{lemma}
\begin{proof}
We will proceed the proof in 3 cases to show that for any $(N_k,A_k)\in\mathscr{N}_{k}$, there is a paring $(N_{k-1},A_{k-1})\in\mathscr{N}_{k-1}$ satisfying $\psi_k(N_{k-1},A_{k-1})=(N_k,A_k)$. From the definition of $\mathscr{N}_{k}$, we have that $N_{k}\in\NBC(\underline{\mathcal{M}})$ and $\mathcal{M}_{k}:={}_{-A_k}\b(\mathcal{M}\backslash N_k^c/N_k\b)$ is acyclic.\\
~\\
{\rm Case 1.} When $e_k\in N_k$, take
\[N_{k-1}=N_k-\{e_k\}\quad \And \quad A_{k-1}=A_k.\]
Obviously, we have $N_{k-1}\subseteq E_{k-1}$ and $A_{k-1}\subseteq E-E_{k-1}$. Moreover, $N_{k}\in\NBC(\underline{\mathcal{M}})$ implies $N_{k-1}\in\NBC(\underline{\mathcal{M}})$. We claim that $\udl{\cal{M}_{k-1}}:=\underline{\mathcal{M}}\backslash N_{k-1}^c/N_{k-1}$ is loopless.  Indeed, exercise 8 in \cite[P. 105]{Oxley2011} states that $\underline{\mathcal{M}}\backslash N_{k-1}^c/N_{k-1}$ is loopless if and only if $N_{k-1}$ is a flat of $\underline{\mathcal{M}}\backslash N_{k-1}^c$. Suppose $N_{k-1}$ is not a flat of $\underline{\mathcal{M}}\backslash N^c_{k-1}$. Then for some element $e_l\notin E_{k-1}$, there is a circuit $Y$ of $\underline{\mathcal{M}}\backslash N_{k-1}^c$ such that $e_l\in Y\subseteq N_{k-1}\cup\{e_l\}$. It implies that $N_{k-1}$ contains a broken circuit $Y-\{e_l\}$ of $\underline{\mathcal{M}}$, a contradiction with $N_{k-1}\in\NBC(\underline{M})$, which proves the claim. Note that $\mathcal{M}_k=\mathcal{M}_{k-1}/\{e_k\}$ and $(e_k,\emptyset)$ is not a signed circuit of $\cal{M}_{k-1}$ since $\udl{\cal{M}_{k-1}}$ is loopless. Immediately from \autoref{acyclic-contraction}, $\mathcal{M}_{k-1}$ is acyclic. Likewise, we have ${}_{-e_k}\mathcal{M}_{k-1}$  is acyclic since $\big({}_{-e_k}\mathcal{M}_{k-1}\big)/\{e_k\}=\mathcal{M}_{k-1}/\{e_k\}=\mathcal{M}_k$.
So we obtain $(N_{k-1},A_{k-1})\in\mathscr{N}_{k-1}$ and $\psi_k(N_{k-1},A_{k-1})=(N_k,A_k)$ from the definition of $\psi_k$, see \eqref{psi_k}. It completes the proof.
\vspace{3mm}\newline
To proceed the other two cases, let \[\mathcal{M}_{k-1}':={}_{-A_k}\big(\mathcal{M}\backslash(E_{k-1}-N_k)/N_k\big).\]
{\rm Case 2.} When $e_k\notin N_k$ and ${}_{-e_k}\mathcal{M}_{k-1}'$ is not acyclic,  take
\[
N_{k-1}=N_k\quad\And\quad A_{k-1}=A_k.
\]
Clearly, we have $N_{k-1}\subseteq E_{k-1}$, $A_{k-1}\subseteq E-E_{k-1}$. Applying similar arguments as Case 1, we can obtain that $N_{k-1}\in\NBC(\underline{\mathcal{M}})$ and  $\underline{\mathcal{M}_{k-1}}$ is loopless.  Note  $\mathcal{M}_{k-1}:={}_{-A_{k-1}}\b(\mathcal{M}\backslash N_{k-1}^c/N_{k-1}\b)$. We have
\[
\mathcal{M}_{k-1}=\mathcal{M}_{k-1}'\quad\And\quad
\mathcal{M}_k=\mathcal{M}_{k-1}\backslash\{e_k\}={}_{-e_k}\mathcal{M}_{k-1}\backslash\{e_k\}.
\]
Immediately, we have  ${}_{-e_k}\mathcal{M}_{k-1}$ is not acyclic. Let $Y=(\underline{Y},\emptyset)$ be a positive circuit of ${}_{-e_k}\mathcal{M}_{k-1}$. The assumption that $\mathcal{M}_k={}_{-e_k}\mathcal{M}_{k-1}\backslash\{e_k\}$ is acyclic implies $e_k\in\underline{Y}$. Next we show that $\mathcal{M}_{k-1}$ is acyclic by contradiction. Suppose $Z=(\underline{Z},\emptyset)$ is a positive circuit of $\mathcal{M}_{k-1}$. Again we have $e_k\in\underline{Z}$ since $\mathcal{M}_k=\mathcal{M}_{k-1}\backslash\{e_k\}$ is acyclic. Moreover, $\underline{Z}-\{e_k\}\ne \emptyset$ since $\underline{\mathcal{M}_{k-1}}$ is loopless.  Let   $Y_1={}_{-e_k}Z=(Y_1^+, e_k)$. Now we have two signed circuits $Y$ and $Y_1$ of ${}_{-e_k}\mathcal{M}_{k-1}$ satisfying $Y\ne -Y_1$ and $e_k\in Y^+\cap Y_1^-$. The axiom $(\cal{C}4)$ yields to a signed circuit $Y_2$ with $Y_2^-\subseteq Y^-\cup Y_1^--\{e_k\}=\emptyset$, i.e., indeed a positive circuit of ${}_{-e_k}\mathcal{M}_{k-1}$. Since $e_k\notin\underline{Y_2}$ and $\mathcal{M}_k={}_{-e_k}\mathcal{M}_{k-1}\backslash\{e_k\}$, $Y_2$ is also a positive circuit of $\mathcal{M}_k$, contradicting that $\mathcal{M}_k$ is acyclic. So we have shown that $(N_{k-1},A_{k-1})\in \mathscr{N}_{k-1}$ and ${}_{-e_k}\mathcal{M}_{k-1}$ is not acyclic.  From the definition of $\psi_k$, we have $\psi_k(N_{k-1},A_{k-1})=(N_k,A_k)$.
\vspace{3mm}\newline
{\rm Case 3.} When $e_k\notin N_k$ and ${}_{-e_k}\mathcal{M}_{k-1}'$ is acyclic,  take
\[
N_{k-1}=N_k\quad\And\quad A_{k-1}=A_k\cup\{e_k\}.
\]
Immediately, we have $N_{k-1}\subseteq E_{k-1}$, $A_{k-1}\subseteq E-E_{k-1}$,  $N_{k-1}\in\NBC(\underline{\mathcal{M}})$, and $\mathcal{M}_{k-1}={}_{-e_k}\mathcal{M}_{k-1}'$, which implies $(N_{k-1},A_{k-1})\in \mathscr{N}_{k-1}$ and $\psi_k(N_{k-1},A_{k-1})=(N_k,A_k)$.
\end{proof}
\begin{remark}\label{inverse}
The arguments of \autoref{lemma10} provide a direct way to obtain the inverse map of $\psi_k$. More precisely, the inverse $\psi_k^{-1}:\mathscr{N}_k\to\mathscr{N}_{k-1}$ is defined as follows. For any $(N_k,A_k)\in \mathscr{N}_k$, let \[\mathcal{M}_{k-1}'={}_{-A_k}\big(\mathcal{M}\backslash(E_{k-1}-N_k)/N_k\big)\]
and define
\[
\psi_k^{-1}(N_k,A_k)=
\begin{cases}
\b(N_k-\{e_k\},A_k\b), & \mbox{if\;} e_k\in N_k;\\
\b(N_k,A_k\b), & \mbox{if\;} e_k\notin N_k\mbox{\;and\;}{}_{-e_k}\mathcal{M}_{k-1}'\mbox{\;is not acyclic}; \\
\b(N_k,A_k\cup\{e_k\}\b), & \mbox{if\;} e_k\notin N_k\mbox{\;and\;}{}_{-e_k}\mathcal{M}_{k-1}'\mbox{\;is acyclic}.
\end{cases}
\]
\end{remark}
\section{Bijection from regions to affine NBC subsets}\label{sec3}
In this section, the bijection \eqref{psi_k} for oriented matroids will be specialized for real hyperplane arrangements to establish a bijection between regions and NBC subsets.
\subsection{Bijection for real linear arrangements}\label{sec3.1}
In order to be compatible with matroid notations, in this section we make the abuse of the notation $\cal{A}$ to denote a multi-arrangement, i.e., a multi-set of hyperplanes.  Let $\cal{A}$ be a linear multi-arrangement in $\mathbb{R}^n$. It is well known that there is a natural way to define an oriented matroid $\mathcal{M}(\mathcal{A})$ such that all regions of $\cal{A}$ are one-to-one corresponding to acyclic reorientations of $\mathcal{M}(\mathcal{A})$. More precisely,  given a finite set $E$ and for each $e\in E$,  take $\bm\alpha_e\in \mathbb{R}^n$ and define
$H_e:\bm\alpha_e\cdot\bm x=0$. This defines a linear multi-arrangement $\cal{A}=\{H_e\mid e\in E\}$. For each hyperplane $H_e:\bm\alpha_e\cdot{\bm x}=0$ in  $\mathbb{R}^n$, the positive and negative sides of $H_e$ can be defined as
\[  H_e^+:\bm\alpha_e\cdot{\bm x}>0 \quad\And\quad H_e^-:\bm\alpha_e\cdot{\bm x}<0 \]
respectively. Recall from \autoref{sec1} that if  $\mathcal{A}_{\sst S}:=\{H_e:e\in S\}$, for some subset $S\subseteq E$, is a circuit of $\mathcal{A}$, there is a linear relation $\sum_{e\in S}\lambda_e\bm\alpha_e=0$ with $\lambda_e\ne 0$. Associated with the circuit $\mathcal{A}_{\sst S}$, we have a signed set $X_{\sst S}=(X_{\sst S}^+,X_{\sst S}^-)$ satisfying
\[
X_{\sst S}^+=\b\{e\in S:\lambda_e>0\b\}\quad \And \quad X_{\sst S}^-=\b\{e\in S:\lambda_e<0\b\}.
\]
Then we define the oriented matroid $\mathcal{M}(\mathcal{A})=(E, \cal{C})$, where $\mathcal{C}$ is the collection of signed sets $X_{\sst S}$ for those $S\subseteq E$ so that $\mathcal{A}_{\sst S}$ is a circuit of $\cal{A}$. The deletion $\mathcal{A}\backslash H_e$ is a multi-arrangement in $\mathbb{R}^n$ by deleting $H_e$ from $\cal{A}$ and the contraction $\mathcal{A}/H_e$ is a multi-arrangement in $H_e$ consisting of  hyperplanes $H\cap H_e$ for all $H\in \mathcal{A}\backslash H_e$. With these definitions, it is easily seen that
\[
\mathcal{M}(\mathcal{A}/H_e)=\mathcal{M}(\mathcal{A})/\{e\}\quad\And\quad \mathcal{M}(\mathcal{A}\backslash H_e)=\mathcal{M}(\mathcal{A})\backslash\{e\},
\]
where the positive (negative) sides of $\mathcal{A}\backslash H_e$  and $\mathcal{A}/H_e$ are inherited from $\mathcal{A}$. More generally, for any two disjoint subsets $S$ and  $T$ of $E$, we have
\[
\mathcal{M}(\mathcal{A}\backslash \cal{A}_{\sst S}/\cal{A}_{\sst T})=\mathcal{M}(\mathcal{A})\backslash S/T.
\]
If $\cal{A}$ and $\mathcal{M}(\mathcal{A})$ share the same total order, it is clear that for $S\subseteq E$, we have
\begin{equation}\label{two-NBC}
S\in  \NBC\b(\udl{\cal{M}(\cal{A})}\b) \quad\Leftrightarrow\quad \cal{A}_{\sst S}\in \NBC(\cal{A}).
\end{equation}
Note that with the choices of normal vectors $\pm{\bm \alpha}_e$, we have distinguished oriented matroids. In fact, for $A\subseteq E$ and each $e\in A$, choosing $-{\bm \alpha}_e$ as the normal vector of $H_e$, then we obtain the oriented matroid ${}_{-A}\mathcal{M}(\mathcal{A})$, a reorientation of $\mathcal{M}(\mathcal{A})$ by $A$.  Recall that a region $\Delta\in \mathcal{R}(\mathcal{A})$ is a connected component of $\mathbb{R}^n-\bigcup_{H\in \mathcal{A}}H$, which implies all points of $\Delta$ lie in the same side of $H_e$ for each $H_e\in \cal{A}$. Let
$A(\Delta)$ be a subset of $E$ defined by
\begin{equation*}
A(\Delta)=\{e\in E\mid \Delta\subseteq H_e^-\}.
\end{equation*}
It is clear from \cite{Bjorner1999} that the reorientation
${}_{-A(\Delta)}\mathcal{M}(\mathcal{A})$ is acyclic. Moreover, the map
\begin{equation}\label{region-acyclic}
A:\cal{R}(\cal{A})\to A(\cal{M}(\cal{A})), \quad \Delta\mapsto A(\Delta)
\end{equation}
is a bijection. In \autoref{Sec2.2}, we have established a bijection from $A(\cal{M}(\cal{A}))$ to $\NBC(\cal{M}(\cal{A}))$. With the obvious correspondence between $\NBC(\cal{M}(\cal{A}))$ and $\NBC(\cal{A})$,
the composition of these two bijections will lead to a bijection from $\cal{R}(\cal{A})$ to $\NBC(\cal{A})$, which is our main result in this section and will be formulated precisely as follows.

Suppose $E=\{e_1, \ldots, e_m\}$ is totally ordered by $e_1\prec\cdots\prec e_m$, and $\cal{A}=\b\{H_{e_i}:\bm\alpha_{e_i}\cdot{\bm x}=0,\; e_i\in E\b\}$. Let $\mathcal{A}_k=\b\{H_{e_i}\mid i\in[k]\b\}$ and $\mathcal{A}_k^c:=\mathcal{A}-\mathcal{A}_k$. For some $\mathcal{B}\subseteq\mathcal{A}$, $\mathcal{A}_k^c/\mathcal{B}$ is a hyperplane arrangement in the affine subspace $\cap \mathcal{B}=\cap_{H\in \cal{B}}H$ defined by
\[\mathcal{A}_k^c/\mathcal{B}:=\b\{(\cap\mathcal{B})\cap H_i\mid H_i\in\mathcal{A}_k^c,\;(\cap\mathcal{B})\cap H_i\ne\cap\mathcal{B} \And\emptyset\b\}.\]
For $k=0,\ldots,m$, let $\mathscr{B}_k$ be the collection of pairings of NBC subsets and regions of $\cal{A}$ defined by
\[
\mathscr{B}_k=\b\{(\mathcal{B}_k,\Delta_k)\mid \mathcal{B}_k\subseteq\mathcal{A}_k,\,\mathcal{B}_k\in\NBC(\mathcal{A})\And \Delta_k\in\mathcal{R}(\mathcal{A}_k^c/\mathcal{B}_k)\b\}.
\]
In particular, $\mathscr{B}_0=\b\{(\emptyset,\Delta_0)\mid \Delta_0\in\mathcal{R}(\mathcal{A})\b\}$ and $\mathscr{B}_m=\b\{(\mathcal{B}_m,\cap\mathcal{B}_m)\mid\mathcal{B}_m\in\NBC(\mathcal{A})\b\}$.  The bijection $\phi_k:\mathscr{B}_{k-1}\to\mathscr{B}_k$ is defined as follows. For any $(\mathcal{B}_{k-1},\Delta_{k-1})\in \mathscr{B}_{k-1}$,  let
\begin{eqnarray}\label{phi_k}
\phi_k(\mathcal{B}_{k-1},\Delta_{k-1})=
\begin{cases}
(\mathcal{B}_{k-1}\cup \{H_{e_k}\}, \Delta_k'\cap H_{e_k}) & \mbox{if\;}\Delta_{k-1}\subseteq H_{e_k}^+\And\Delta_{k-1}\ne\Delta_k'; \\
(\mathcal{B}_{k-1},\Delta_{k-1}) & \mbox{if\;}\Delta_{k-1}\subseteq H_{e_k}^+\And \Delta_{k-1}=\Delta_k';\\
(\mathcal{B}_{k-1},\Delta_k') & \mbox{if\;}\Delta_{k-1}\subseteq H_{e_k}^-,
\end{cases}
\end{eqnarray}
where $\Delta_k'$ is the unique region in $\mathcal{R}(\mathcal{A}_k^c/\mathcal{B}_{k-1})$ containing $\Delta_{k-1}$, since  $\Delta_{k-1}\in \mathcal{R}(\mathcal{A}_{k-1}^c/\mathcal{B}_{k-1})$ and \[\mathcal{A}_k^c/\mathcal{B}_{k-1}=(\mathcal{A}_{k-1}^c\backslash H_{e_k})/\mathcal{B}_{k-1}.\]
Likewise, the composition
\[\phi:=\phi_m\circ\phi_{m-1}\circ\cdots\circ\phi_1\]
is a bijection from $\mathscr{B}_0$ to $\mathscr{B}_m$, which clearly induces a bijection from $\cal{R}(\mathcal{A})$ to $\NBC(\mathcal{A})$.

\subsection{Verification on the bijection $\phi_k$}
Let $\cal{A}=\b\{H_{e_i}:\bm\alpha_{e_i}\cdot{\bm x}=0,\; e_i\in E\b\}$ be a real hyperplane arrangement and $\cal{M}(\cal{A})$ the associated oriented matroid over the ordered set $\{e_1\prec\cdots\prec e_m\}$. Let the set $\mathscr{N}_k$ and the bijection $\psi_k:\mathscr{N}_{k-1}\to \mathscr{N}_k$ be defined for $\cal{M}(\cal{A})$  as \eqref{psi_k}, and let the set $\mathscr{B}_k$ and the map $\phi_k:\mathscr{B}_{k-1}\to \mathscr{B}_k$ defined for $\cal{A}$ as \eqref{phi_k}. It is clear that the correspondences \eqref{two-NBC} and \eqref{region-acyclic} will lead to a natural bijection $\tau_k: \mathscr{B}_k\to \mathscr{N}_k$, where $\tau_k(\cal{B}_k,\Delta_k)=(N_k, A_k)$  is defined by
\[
N_k=\{e\mid H_e\in\mathcal{B}_k\}\quad\And\quad A_k=\{e\in E-E_k\mid \Delta_k\subseteq H_e^-\}.
\]
Under this bijection, the map $\phi_k$ of \eqref{phi_k} will be identified with the bijection $\psi_k$ of \eqref{psi_k} case by case in the sense $\tau_k\circ\phi_k=\psi_k\circ\tau_{k-1}$, i.e., the diagram below commutes. Thus  $\phi_k$ is bijective.
\begin{figure}[H]
\centering
\begin{tikzpicture}
\draw[->] (0.5,0)--(2,0) node[right]{$\mathscr{N}_k$};
\draw[<-] (0,0.25)--(0,1.5);
\draw[->] (0.5,1.75)--(2,1.75) node[right]{$\mathscr{B}_k$};
\draw[->] (2.35,1.5)--(2.35,0.25);
\draw (0,0) node {$\mathscr{N}_{k-1}$};
\draw (0,1.75) node {$\mathscr{B}_{k-1}$};
\draw (0,1)  node[left] {$\tau_{k-1}$};
\draw (2.25,1)  node[right]  {$\tau_k$};
\draw (1.25,0)  node[above] {$\psi_k$};
\draw (1.25,1.75) node[above] {$\phi_k$};
\end{tikzpicture}
\end{figure}

Next we will explain the detailed identification of $\phi_k$ with $\psi_k$. Since the arguments for other two cases are similar, we will only deal with the first case in the definitions  $\eqref{psi_k}$ and $ \eqref{phi_k}$. Namely, if $(\mathcal{B}_{k-1},\Delta_{k-1})\in \mathscr{B}_{k-1}$ and  $\Delta_k'$ is the unique region in $\mathcal{R}(\mathcal{A}_k^c/\mathcal{B}_{k-1})$ with $\Delta_{k-1}\subseteq \Delta_{k}'$, we will show that the following diagram holds for $\Delta_{k-1}\subseteq H_{e_k}^+$ and $\Delta_{k-1}\neq\Delta_k'$.
\begin{figure}[H]
\centering
\begin{tikzpicture}
\draw[|->] (1.5,0)--(3.5,0) node[right]{$(N_{k-1}\cup\{e_k\},A_{k-1})$};
\draw[<-|] (0.15,0.25)--(0.15,1.5);
\draw[|->] (1.5,1.75)--(3.5,1.75) node[right]{$(\mathcal{B}_{k-1}\cup\{H_{e_k}\},\Delta_k'\cap H_{e_k})$};
\draw[<-|] (5.5,0.25)--(5.5,1.5);
\draw (0.15,0) node {$(N_{k-1},A_{k-1})$};
\draw (0.15,1.75) node {$(\mathcal{B}_{k-1},\Delta_{k-1})$};
\draw (0.15,1)  node[left] {$\tau_{k-1}$};
\draw (5.5,1)  node[right]  {$\tau_{k}$};
\draw (2.5,0)  node[above] {$\psi_k$};
\draw (2.5,1.75) node[above] {$\phi_k$};
\end{tikzpicture}
\end{figure}

It will be enough to prove that if $\Delta_{k-1}\subseteq H_{e_k}^+$ and $\Delta_{k-1}\neq\Delta_k'$, then $(N_{k-1},A_{k-1})$ meets the first case in \eqref{psi_k} and $\tau_k(\cal{B}_{k-1}\cup \{H_{e_k}\},\Delta_k'\cap H_{e_k})=(N_{k-1}\cup\{e_k\},A_{k-1})$, i.e., in this case we have
\begin{itemize}
\item[(1)] $ A_{k-1}=\{e\in E-E_k\mid \Delta_k'\cap H_{e_k}\subseteq H_e^-\}$;
\item[(2)] $e_k\notin A_{k-1}$ and ${}_{-e_k}\mathcal{M}(\cal{A})_{k-1}$ is acyclic.
\end{itemize}
Notice that the multi-arrangement $\mathcal{A}_{k-1}^c/\mathcal{B}_{k-1}$ is obtained from the multi-arrangement $\mathcal{A}_k^c/\mathcal{B}_{k-1}$ by adding $(\cap \mathcal{B}_{k-1})\cap H_{e_k}$. From the assumptions $\Delta_{k-1}\in\cal{R}(\mathcal{A}_{k-1}^c/\mathcal{B}_{k-1})$, $\Delta_k'\in \cal{R}(\mathcal{A}_k^c/\mathcal{B}_{k-1})$, $\Delta_{k-1}\subsetneqq\Delta_k'$, and $\Delta_{k-1}\subseteq H_{e_k}^+$, we can see that the region $\Delta_k'$ is separated by $H_{e_k}$ into two regions of $\mathcal{A}_{k-1}^c/\mathcal{B}_{k-1}$ and one of which is $\Delta_{k-1}=\Delta_k'\cap H_{e_k}^+$. It means that for $e\in E-E_k$ (obviously, $e\ne e_k$), the followings are equivalent
\[
\Delta_{k-1}\subseteq H_{e}^-\quad\Leftrightarrow\quad\Delta_k'\subseteq H_{e}^-\quad\Leftrightarrow\quad \Delta_k'\cap H_{e_k}\subseteq H_{e}^-.
\]
Again by $\Delta_{k-1}\subseteq H_{e_k}^+$, we have $e_k\notin A_{k-1}$ and
\[
A_{k-1}:=\{e\in E-E_{k-1}\mid \Delta_{k-1}\subseteq H_e^-\}=\{e\in E-E_k\mid \Delta_k'\cap H_{e_k}\subseteq H_e^-\},
\]
which proves (1) and the first part of (2). It remains to show that the oriented matroid \[{}_{-e_k}\mathcal{M}(\cal{A})_{k-1}={}_{-A_{k-1}\sqcup\{e_k\}}\b(\mathcal{M}(\cal{A})\backslash N_{k-1}^c/N_{k-1}\b)\]
is acyclic. It is clear from the definition of $\cal{M}(\cal{A})$ that
\[\mathcal{M}(\cal{A})\backslash N_{k-1}^c/N_{k-1}=\cal{M}(\cal{A}_{k-1}^c/\cal{B}_{k-1}).\]
Note that both $\Delta_k'\cap H_{e_k}^+$ and $\Delta_k'\cap H_{e_k}^-$ are regions of $\cal{A}_{k-1}^c/\cal{B}_{k-1}$. Since $A(\Delta_{k-1})=A(\Delta_k'\cap H_{e_k}^+)=A_{k-1}$, we have $A(\Delta_k'\cap H_{e_k}^-)=\{e_k\}\cup A_{k-1}$, which implies ${}_{-e_k}\mathcal{M}(\cal{A})_{k-1}$ is acyclic by the definition of \eqref{region-acyclic}.

\subsection{Extension to real hyperplane arrangements}
In this subsection, the bijection in \autoref{sec3.1} will be extended to real hyperplane arrangements via the coning operation.  Let $\mathcal{A}=\b\{H_i:\bm\alpha_i\cdot\bm x=b_i\mid i\in [m]\b\}$ be a real hyperplane arrangement in $\mathbb{R}^n$. Then its {\em coning} $\c \mathcal{A}$ is a linear arrangement in $\mathbb{R}^{n+1}$ defined by
\[
\c \mathcal{A}=\b\{\c H_i:\bm\alpha_i\cdot\bm x=b_i x_{n+1}\mid i\in [m]\b\}\sqcup\{K_0:x_{n+1}=0\}.
\]
Let $K_1:x_{n+1}=1$. It is clear that
\[
\mathcal{A}=\b\{\c H_i\cap K_1\mid i\in [m]\b\},
\]
i.e., the hyperplane arrangement $\mathcal{A}$ is nothing but the restriction of $\c\mathcal{A}$ on $K_1$. It follows that
\[
\mathcal{R}(\mathcal{A})=\{\Delta\cap K_1\mid \Delta\in\mathcal{R}(\c\mathcal{A})\And \Delta\cap K_1\ne\emptyset\}.
\]
Then the bijection of \autoref{sec3.1} can be applied to the linear arrangement $\c\mathcal{A}$ to obtain the bijection for hyperplane arrangement $\cal{A}$. Suppose $\c\mathcal{A}$ is totally ordered by $\c H_{1}\prec\c H_{2}\prec\cdots\prec\c H_{m}\prec K_0$ and $\mathcal{A}$ is ordered by $H_{1}\prec H_{2}\prec\cdots\prec H_{m}$. Note that for any subset $S\subseteq [m]$, $\mathcal{A}_{\sst S}$ is an affine NBC subset of $\mathcal{A}$ if and only if both  $(\c\mathcal{A})_{\sst S}$ and $(\c\mathcal{A})_{\sst S}\cup\{K_0\}$  are the NBC subsets of $\c\mathcal{A}$. More precisely, the bijection from regions of $\cal{A}$ to its affine subsets is obtained from the bijection of \autoref{sec3.1} as follows. For each $\Delta_1\in \cal{R}(\cal{A})$, there is a unique $\Delta\in \cal{R}(\c\cal{A})$ with $\Delta_1=\Delta\cap K_1$, which is by the bijection of \autoref{sec3.1} sent to an NBC subset of $\c\cal{A}$, denoted $\c\cal{B}$. Then the set $\{H_i\in \cal{A}\mid \c H_i\in \c\cal{B}\}$ is automatically an affine NBC subset of $\cal{A}$, which establishes the bijection. To state the bijection explicitly, we need to introduce the set $\mathscr{B}_k$ and the bijection $\phi_k:\mathscr{B}_{k-1}\to \mathscr{B}_k$ for $\cal{A}$, which can be defined exactly in the same manner as \eqref{phi_k}. Indeed, all notations and definitions for linear arrangements in \eqref{phi_k} are also available for hyperplane arrangements. The following remark is the inverse of $\phi_k$ in \eqref{phi_k} for hyperplane arrangements $\cal{A}$, whose composition gives a bijection from affine NBC subsets of $\cal{A}$ to its regions.
\begin{remark}\label{inverse1} Let $\mathcal{A}=\b\{H_i:\bm\alpha_i\cdot\bm x=b_i\mid i\in [m]\b\}$ be a hyperplane arrangement in $\mathbb{R}^n$ with the total order $H_1\prec\cdots\prec H_m$. Recall from \autoref{sec3.1} that
$\mathscr{B}_k$ is the collection of pairings of NBC subsets and regions of $\cal{A}$ defined by
\[
\mathscr{B}_k=\b\{(\mathcal{B}_k,\Delta_k)\mid \mathcal{B}_k\subseteq\mathcal{A}_k,\,\mathcal{B}_k\in\NBC(\mathcal{A})\And \Delta_k\in\mathcal{R}(\mathcal{A}_k^c/\mathcal{B}_k)\b\}.
\]
For any paring $(\mathcal{B}_k,\Delta_k)\in\mathscr{B}_k$, let $\Delta_k'$ be the unique region in $\mathcal{A}_k^c/(\mathcal{B}_k-\{H_{k}\})$ with $\Delta_k\subseteq \Delta_k'$.
The inverse map $\phi_k^{-1}:\mathscr{B}_k\to\mathscr{B}_{k-1}$ is defined to be
\begin{eqnarray*}\label{phi_k^{-1}}
\phi_k^{-1}(\mathcal{B}_k,\Delta_k)=
\begin{cases}
(\mathcal{B}_k-\{H_{k}\},\Delta_k'\cap H_{k}^+) & \mbox{if\;}\Delta_k\subseteq H_{k};\\
(\mathcal{B}_k,\Delta_k) & \mbox{if\;} \Delta_k\cap H_{k}=\emptyset;\\
(\mathcal{B}_k,\Delta_k'\cap H_{k}^-) & \mbox{if\;} \Delta_k\cap H_{k}\ne\emptyset,\; \Delta_k\nsubseteq H_{k}.
\end{cases}
\end{eqnarray*}
\end{remark}

One should observe that the bijections $\phi_k$ directly give rise to an extended version \autoref{Regions} of \autoref{Region-NBC}, which was first proved in \cite{Fu-Wang2024} by induction on the number of hyperplanes. In particular,  this coincides with \autoref{Region-NBC} at $k=m$. 
\begin{theorem}[\cite{Fu-Wang2024}, Theorem 4.2]\label{Regions}
Let $\mathcal{A}=\{H_1,H_2,\ldots,H_m\}$ be a hyperplane arrangement in a real vector space. Then for each $k=0,1,\ldots,m$,
\[
|\mathcal{R}(\mathcal{A})|=\sum_{\mathcal{B}\in\NBC(\mathcal{A}),\,\mathcal{B}\subseteq\mathcal{A}_k}
|\mathcal{R}(\mathcal{A}_k^c/\mathcal{B})|=|\NBC(\mathcal{A})|.
\]
\end{theorem}

To end this section, we will give a small example on the bijection from regions to affine NBC subsets.
\begin{example}{\rm
Let
\[\mathcal{A}=\{H_1:y=0,\;H_2:x-y=0,\;H_3:x=0,\;H_4:x+y=1\}\]
be a hyperplane arrangement in $\mathbb{R}^2$. \autoref{figure1}  describes correspondences between $\mathcal{R}(\mathcal{A})$ and $\NBC(\mathcal{A})$, i.e., for each region of $\cal{A}$, the corresponding affine NBC subset is given inside the region. E.g., for the yellow region $\Delta_0$ in \autoref{figure1} defined by $\Delta_0=H_1^-\cap H_2^+\cap H_3^+\cap H_4^+$, the corresponding affine NBC subset is $\{H_2,H_4\}$. The complete algorithm from $\Delta_0$ to $\{H_2,H_4\}$ is the following.
\begin{figure}[H]
\centering
\begin{tikzpicture}[scale=1,line width=1pt]
\fill[yellow] (3,0)--(5,0)--(5,-2);
\draw (-2,0)--(5,0)node[right]{$H_1:y=0$};
\draw (0,-2)--(0,5)node[above]{$H_3:x=0$};
\draw (-2,-2)--(5,5)node[right]{$H_2:x-y=0$};
\draw (5,-2)--(-2,5)node[left]{$H_4:x+y=1$};
\node at (1.5,.5){$\{H_1,H_3\}$};
\node at (.5,1.5){$\{H_3\}$};
\node at (1.5,3){$\{H_3,H_4\}$};
\node at (3,1.5){$\{H_1,H_4\}$};
\node at (-.5,4.25){$\{H_4\}$};
\node at (-.5,1.5){$\{H_1\}$};
\node at (-.6,-.25){$\emptyset$};
\node at (-.5,-1.25){$\{H_2\}$};
\node at (1.5,-1.25){$\{H_2,H_3\}$};
\node at (4.35,-.5){$\{H_2,H_4\}$};
\node at (0.25,3){$P_2$};
\node at (0.35,4.25){$Q_2$};
\node at (1.8,1.5){$P_1$};
\node at (3.3,3){$Q_1$};
\node at (3,0.3){$P_0$};
\end{tikzpicture}
\hspace{2cm}\caption{The bijection: \;$\mathcal{R}(\mathcal{A})\to \NBC(\mathcal{A})$}
\label{figure1}
\end{figure}
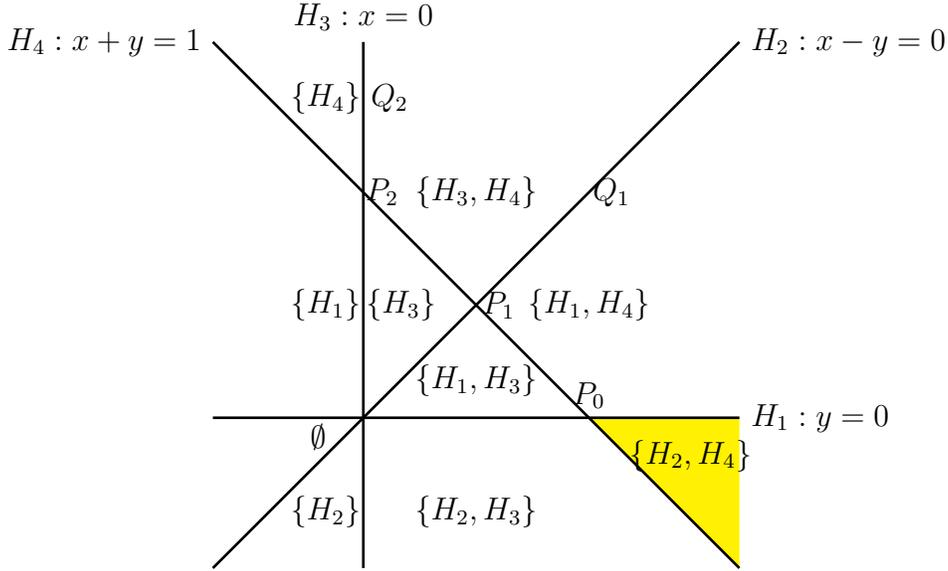

\begin{itemize}
   \item [Step 1.] For $k=1$, we have $\cal{B}_{k-1}=\emptyset$ and $\cal{A}_k^c=\{H_2,H_3,H_4\}$. So $\cal{A}_{k}^c/\cal{B}_{k-1}=\cal{A}\backslash \{H_1\}$ in $\mathbb{R}^2$. Start with $(\cal{B}_0,\Delta_0)=(\emptyset,\Delta_0)$. The unique region $\Delta_1'$ of $\cal{A}_{1}^c/\cal{B}_{0}$ containing $\Delta_0$ is the angular region $\angle Q_1P_1P_0$. Since $\Delta_0\subseteq H_1^-$, then $\phi_1(\cal{B}_0,\Delta_0)=(\emptyset,\Delta_1')$ by the third case of the definition \eqref{phi_k}.
   \item [Step 2.] For $k=2$, we have $\mathcal{B}_{k-1}=\emptyset$ and $\mathcal{A}_k^c=\{H_3,H_4\}$. So  $\cal{A}_{k}^c/\cal{B}_{k-1}=\cal{A}\backslash \{H_1,H_2\}$ in $\mathbb{R}^2$. Let $(\cal{B}_1,\Delta_1)=(\emptyset,\Delta_1')$.  The unique region $\Delta_2'$ of $\cal{A}_{2}^c/\cal{B}_{1}$ containing $\Delta_1$ is the angular region $\angle Q_2P_2P_0$. Obviously, $\Delta_1\subseteq H_2^+$ and $\Delta_1\ne\Delta_2'$. Then from the first case of \eqref{phi_k}, we have $\phi_2(\cal{B}_1,\Delta_1)=(\{H_2\},\Delta_2'\cap H_2)$, where $\Delta_2'\cap H_2$ is the open half line $P_1Q_1$.
   \item [Step 3.] For $k=3$, we have $\mathcal{B}_{k-1}=\{H_2\}$ and $\mathcal{A}_k^c=\{H_4\}$. So  $\cal{A}_{k}^c/\cal{B}_{k-1}=\{H_4\cap H_2\}$ in $H_2$. Consider $(\cal{B}_2,\Delta_2)=(\{H_2\},\Delta_2'\cap H_2)$. Clearly, $\Delta_2\subseteq H_3^+$ and the unique region $\Delta_3'$ of $\cal{A}_{3}^c/\cal{B}_{2}$ containing $\Delta_2$ is $\Delta_2$ itself, i.e., $\Delta_3'=\Delta_2$. By the second case of \eqref{phi_k}, we obtain $\phi_3(\cal{B}_2,\Delta_2)=(\{H_2\},\Delta_2'\cap H_2)$.
   \item [Step 4.]  For $k=4$, we have $\mathcal{B}_{k-1}=\{H_2\}$ and $\mathcal{A}_k^c=\emptyset$. So $\mathcal{A}_k^c/\mathcal{B}_{k-1}=\emptyset$ in $H_2$. So far we have $(\cal{B}_3,\Delta_3)=(\{H_2\},\Delta_2'\cap H_2)$. The unique region $\Delta_4'$ of $\cal{A}_{4}^c/\cal{B}_{3}$ containing $\Delta_3$ is the whole space $H_2$. Clearly, we have $\Delta_3\subseteq H_4^+$ and $\Delta_3\ne \Delta_4'$. Then from the first case of \eqref{phi_k},  we have $\phi_4(\cal{B}_3,\Delta_3)=(\{H_2,H_4\},H_2\cap H_4)$, i.e., $\cal{B}_4=\{H_2,H_4\}$.
 \end{itemize}
 }
\end{example}
\section*{Acknowledgements}
The first author is supported by National Natural Science Foundation of China under Grant No. 12301424. The last author is supported by Hunan Provincial Innovation Foundation for Postgraduate under Grant No. CX20220434.


\begin{thebibliography}{99}
\bibitem{Birkhoff1912}
G. D. Birkhoff, A determinant formula for the number of ways of coloring a map, Ann. of Math. 14 (1912), 42--46.

\bibitem{Bjorner1999}
A. Bj\"orner, M. Las Vergnas, B. Sturmfels, N. White, G. Ziegler, Oriented Matroids, Second edition, Cambridge University Press, New York, 1999.

\bibitem{Bjorner1991}
A. Bj\"{o}rner, G. M. Ziegler, Broken circuit complexes: factorizations and generalizations, J. Combin. Theory Ser. B 51 (1991),  96--126.

\bibitem{Bland-Lasvergnas1978}
R. G. Bland, M. Las Vergnas, Orientability of matroids, J. Combin. Theory Ser. B 24 (1978), 94--123.

\bibitem{Blass-Sagan1986}
A. Blass, B. E. Sagan, Bijective proofs of two broken circuit theorems, J. Graph Theory 10 (1986), 15--21.

\bibitem{Brylawski-Oxley1981}
T. Brylawski, J. Oxley, The broken-circuit complex: its structure and factorizations, European J. Combin. 2 (1981), 107--121.

\bibitem{Delucchi2008}
E. Delucchi, Shelling-type orderings of regular CW-complexes and acyclic matchings of the Salvetti complex, Int. Math. Res. Not. IMRN (2008), Art. ID rnm167, 39 pp.

\bibitem{Fu-Wang2024}
H. Fu, S. Wang, J. Yang, New perspectives on polynomial invariants, Discrete Math. 347 (2024), 114101.

\bibitem{Gioan2002}
E. Gioan, Correspondance naturelle entre bases et rorientations des matrodes orient\'{e}s, PhD thesis, University of Bordeaux 1 (2002).

\bibitem{Gioan-Vergnas2006}
E. Gioan, M. Las Vergnas, The active bijection between regions and simplices in supersolvable arrangements of hyperplanes, Electron. J. Combin. 11 (2006), no. 2, 39 pp.

\bibitem{Gioan-Vergnas2007}
E. Gioan, M. Las Vergnas, Fully optimal bases and the active bijection in graphs, hyperplane arrangements, and oriented matroids, Electronic Notes in Discrete Mathematics 29 (2007), 365--371.

\bibitem{Gioan-Vergnas2009}
E. Gioan, M. Las Vergnas, The active bijection in graphs, hyperplane arrangements, and oriented matroids. I. The fully optimal basis of a bounded region, European J. Combin. 30 (2009), 1868--1886.

\bibitem{Gioan-Vergnas2018}
E. Gioan, M. Las Vergnas, The active bijection 2.b- Decomposition of activities for oriented matroids, and general definitions of the active bijection, (2018), arXiv:1807.06578.

\bibitem{Gioan-Vergnas2019}
E. Gioan, M. Las Vergnas, The active bijection for graphs, Adv. in Appl. Math. 104 (2019), 165--236.

\bibitem{Jambu-Terao1986}
M. Jambu, H. Terao, Arrangements of hyperplanes and broken circuits, Singularities (Iowa City, I. A, 1986), 147--162,
Contemp. Math., 90, Amer. Math. Soc., Providence, R. I.,  1989.

\bibitem{Jewell-Orlik2002}
K. Jewell, P. Orlik, Geometric relationship between cohomology of the complement of real and complexified arrangements, Topology Appl.  118 (2002), 113--129.

\bibitem{LasVergnas1980}
M. Las Vergnas, Convexity in oriented matroids, J. Combin. Theory Ser. B 29 (1980), 231--243.

\bibitem{Orlik1992}
P. Orlik, H. Terao, Arrangement of Hyperplane, Springer-Verlag, Berlin, 1992.

\bibitem{Oxley2011}
J. Oxley, Matroid Theory, Second edition, Oxford University Press, New York, 2011.

\bibitem{Rota1964}
G.-C. Rota, On the foundations of combinatorial theory. I. Theory of M\"{o}bius functions, Z. Wahrscheinlichkeitstheorie und Verw. Gebiete, 2 (1964), 340--368.

\bibitem{Sagan1995}
B. E. Sagan, A generalization of Rota's NBC theorem, Adv. Math. 111 (1995), 195--207.

\bibitem{Stanley1973}
R. P. Stanley,  Acyclic orientations of graphs, Discrete Math. 5 (1973), 171--178.

\bibitem{Stanley2007}
R. P. Stanley, An introduction to hyperplane arrangements, in: E. Miller, V. Reiner, B. Sturmfels (Eds.), Geometric Combinatorics, in: IAS/Park City Math. Ser., vol. 13, Amer. Math. Soc., Providence, R. I., 2007, pp. 389--496.

\bibitem{Whitney1932}
H. Whitney, A logical expansion in mathematics, Bull. Amer. Math. Soc. 38 (1932), 572--579.

\bibitem{Whitney1935}
H. Whitney, On the abstract properties of linear dependence, Amer. J. Math. 57 (1935), 509--533.

\bibitem{Zaslavsky1975}
T. Zaslavsky, Facing up to arrangements: face-count formulas for partitions of space by hyperplanes, Mem. Amer. Math. Soc. vol. 1. no. 154 (1975).
\end{thebibliography}
\end{document}